\documentclass[a4paper, 11pt]{article}

\usepackage{color}
\def\rrr{}
\def\bbb{}

\usepackage{color}
\def\rrr{}
\def\bbb{}

\def\mychange#1{#1}
\def\mchange#1{#1}

\usepackage{amsmath,amssymb,amsthm}
\usepackage{cleveref}
\usepackage{cases}
\usepackage{mathrsfs}
\usepackage{amsfonts}
\allowdisplaybreaks[4]
\usepackage{amscd}
\usepackage[all]{xy}
\usepackage{comment}
\usepackage{graphicx}
\usepackage[top=30truemm,bottom=30truemm,left=25truemm,right=25truemm]{geometry}

\usepackage{url}

 \makeatletter
    
    \@addtoreset{equation}{section}
  \makeatother

\theoremstyle{plain}
\newtheorem{thm}{Theorem}[section]
\newtheorem{lem}[thm]{Lemma}
\newtheorem{prop}[thm]{Proposition}
\newtheorem{cor}[thm]{Corollary}
\newtheorem{exa}[thm]{Example}

\DeclareMathOperator{\Ker}{Ker}

\DeclareMathOperator{\End}{End}

\DeclareMathOperator{\Homo}{H}

\DeclareMathOperator{\DR}{DR}

\newcommand{\R}{\mathbb{R}}
\newcommand{\D}{\mathbb{D}}
\newcommand{\C}{\mathbb{C}}
\newcommand{\DD}{\mathcal{D}}

\newcommand{\A}{\mathbb{A}}
\newcommand{\Q}{\mathbb{Q}}
\newcommand{\Z}{\mathbb{Z}}

\newcommand{\Gm}{\mathbb{G}_m}

\newcommand{\ii}{\sqrt{-1}}

\newcommand{\ve}{\varepsilon}

\def\pd#1{\partial_{#1}}

\begin{document}

\title{An algorithm of computing cohomology intersection number of hypergeometric integrals}

\author{Saiei-Jaeyeong Matsubara-Heo\footnote{Department of Mathematics, Graduate School of Science, Kobe  University, 1-1 Rokkodai, Nada-ku, Kobe 657-8501, Japan.\newline e-mail: \texttt{saiei@math.kobe-u.ac.jp}}\and 
Nobuki Takayama\footnote{Department of Mathematics, Graduate School of Science, Kobe  University, Japan.\newline e-mail: \texttt{takayama@math.kobe-u.ac.jp}}
}
\date{}

\maketitle

\begin{abstract}
We show that the cohomology intersection number of a twisted Gauss-Manin connection  with regularization condition is a rational function. As an application, we obtain a new quadratic relation associated to period integrals of a certain family of K3 surfaces.
\end{abstract}

\section{Introduction}
The study of intersection numbers of twisted cohomology groups
and twisted period relations for hypergeometric functions
started with the celebrated work by
Cho and Matsumoto \cite{CM}.
They clarified that the cohomology intersection number appears naturally as a part of the quadratic relation, a class of functional \mchange{identities} of hypergeometric functions.
They also developed a systematic method of computing the cohomology intersection number for $1$-dimensional integrals.
Since this work, several methods have been proposed to evaluate
intersection numbers of twisted cohomology groups,
see, e.g., \cite{Aomoto-Kita}, \cite{Goto-Matsumoto},
\cite{Goto-Matsumoto-2018}, \cite{KY}, \cite{Matsumoto-osaka}, \cite{Ohara-Sugiki-Takayama}
and references therein.
All methods utilize comparison theorems of twisted cohomology groups and residue calculus.

We propose a new method in this paper.
Our method reduces the problem of evaluating the intersection numbers
to the question of finding a rational solution of a system
of linear differential equations.
The key idea of this method is that we regard
the intersection matrix of 
a twisted cohomology group as a horizontal section 
of the tensor product of a connection $E$ and its dual.
For the proof, we assume an important condition: the regularization condition. With the aid of this condition, we can replace \mchange{transcendental objects} such as cohomology groups with compact support by algebraic de Rham cohomology groups.

In the beginning of \S 2, we briefly define the intersection matrix and derive the twisted
period relation in a general framework. The twisted period relation for hypergeometric functions was first given by Cho and Matsumoto \cite{CM}. Calculations of intersection numbers given by Matsumoto \cite{Matsumoto-osaka} will be very helpful to understand what  intersection numbers are and to study hypergeometric functions associated to hyperplane arrangements. A more comprehensive version of a definition of  intersection  numbers and the twisted period relation is given in \cite[\S 3, \S 4]{Ohara-Sugiki-Takayama}. We recommend readers to refer to these papers.

When the twisted cohomology group is associated to the GKZ system (\cite{GKZEuler}) for a matrix $A$ admitting a \mchange{regular} unimodular triangulation,
our method gives a complete algorithm to determine
the intersection matrix with the aid of algorithms of finding
rational solutions of a system of differential equations (see, e.g., \cite{Oaku-Takayama-Tsai} and its references)
and the formula of intersection numbers of twisted homology groups
for GKZ hypergeometric systems \cite{MH}.
For an introductory exposition of the intersection numbers of the twisted homology groups, see the book by Aomoto and Kita \cite[\S 2.3]{Aomoto-Kita} or \cite{KY}.
Our method is demonstrated for the matrix $A$ which appears in a study
of a $K3$ surface \cite{NS01} in the last section.
We note that the computation of the cohomology intersection number of this example has not been obtained from the previous approaches.
Another important advantage of our method is that the validity of the formula of cohomology intersection numbers can be checked by computer algebra systems without any help of an expert. While this paper was under review, we obtained
an algorithm to construct the Pfaffian systems with respect to a given
cohomology basis (\cite{ICMS}). 
We implemented the construction method
and the algorithm of this paper as a Risa/Asir package ``{\tt mt\_gkz.rr}''
(\cite{MTGKZ}).

We give a summary of contents. In \S 2, we state and prove the main theorem of this paper (\cref{thm:main}). In \S 3, we recall the basic set up of twisted cohomology theory associated to GKZ systems and describe our algorithm of computing the cohomology intersection numbers. We also mention a relation between the secondary fan and the common denominator of the cohomology intersection matrix. In \S 4, we demonstrate how our method works for a particular GKZ system which arises in a study of $K3$ surfaces.

The first author is supported by JSPS KAKENHI Grant Number 19K14554, and the second author is supported by JSPS KAKENHI Grant Number 17K05279. Both of the authors are supported by JST CREST Grant Number JP19209317. The first author thanks Masatoshi Noumi for valuable comments. Corollary 2.2 is an outcome of the comments. He would also like to thank Frits Beukers for letting him know the paper \cite{BJ}, where the authors discuss duality relations of univariate hypergeoemtric systems from a similar point of view. We deeply appreciate several comments of the reviewer for improving some explanations and for clarifying main points of our paper.

\section{Statement and the proof}

Let $X$, $Y$ be complex smooth quasi-projective varieties, and $f:X\rightarrow Y$ be an affine morphism. We assume $f$ is generically smooth in order to apply Thom-Mather's 1st isotopy lemma. We put $d=\dim X -\dim Y$. In the following, we use the notation of \cite{Hotta-Takeuchi-Tanisaki}. For any bounded complex of $\DD_X$-modules $L$, we set $\int_fL=\R f_*(\DD_{Y\leftarrow X}\otimes_{f^{-1}\DD_X} L)$, where $\DD_{Y\leftarrow X}$ is the transfer module. We also put $\D L=\R\mathcal{H}om_{\DD_X}(L,\DD_X)\otimes_{\mathcal{O}_X}\Omega_X^{\otimes -1}$. With this aid, we set $\int_{f!}L=\D\circ\int_f\circ\D L$. Let $M=(E,\nabla)$ be a regular integrable connection on $X$, and we put $N=\int_f^0 M=\Homo^0 \int_f M$. By the general theory \mchange{of holonomic $\mathcal{D}$-modules}, we see that $\int_f M$ is a complex of $D$-modules with regular holonomic cohomologies and therefore, $N$ is a regular connection defined on a non-empty Zariski open subset $U$ of $Y$.  By shrinking $U$ if necessary, we may assume that $f:f^{-1}(U)\rightarrow U$ is smooth. Now let us assume further the following non-trivial condition: the canonical morphism
\begin{equation}\label{Regularization}
\int_{f!} M\rightarrow\int_f M,
\end{equation}
is an isomorphism, or equivalently, the canonical morphism
\begin{equation}
\D\int_{f}M\rightarrow\int_f\D M,
\end{equation}
is an isomorphism. Here, $\mathbb{D}$ stands for the holonomic dual.
This condition is called the regularization condition. This name comes from the literature of hypergeometric functions (\cite[Theorem 3.1]{Aomoto-Kita}). Since $M$ is a connection, we see that $\D M$ is isomorphic to the dual connection $(E^\vee,\nabla^\vee)$ (\cite{Hotta-Takeuchi-Tanisaki} Example 2.6.10).

Considering the Spencer resolution, we see that $\int_f M$ is represented on $U$ by $\R f_*(\DR_{X/Y}(M))$ where $\DR_{X/Y}(M)=(\Omega^{\bullet+d}_{X/Y}(E),\nabla_{X/Y})$ is the relative de Rham complex (\cite{DMSS} 1.4. Proposition). Again by shrinking $U$, we may assume that $N$ is free on $U$, and that $f^{an}:f^{-1}(U)^{an}\rightarrow U^{an}$ is a fiber bundle by Thom-Mather's 1st isotopy lemma (\cite[(4.14) Th\'eor\`eme]{Verdier}, \cite[Cor. 1.2.14]{Tibar}) in view of \cite[(2.2) Th\'eor\`eme and (3.3) Th\'eor\`eme]{Verdier}. We take a free basis $\{\phi_j\}_{j=1}^r\subset \Homo^0(U,N)=\Homo^0(\Gamma(f^{-1}(U),\DR_{X/Y}( M)))$ (resp. $\{\psi_j\}_{j=1}^r\subset \Homo^0(U,\D N)$) on $U$. The connection $\nabla^{GM}:N\rightarrow \Omega^1(N)$ (resp. $\nabla^{\vee GM}:\D N\rightarrow \Omega^1(\D N)$) with respect to this basis $\{\phi_j\}_{j=1}^r$ (resp. $\{\psi_j\}_{j=1}^r$) is given by $\nabla^{GM}=d+\Omega$ (resp. $\nabla^{\vee GM}=d+\Omega^\vee$) for some $r\times r$ matrix $\Omega=(\omega_{ij})_{i,j=1}^r$ (resp. $\Omega^\vee=(\omega_{ij}^\vee)_{i,j=1}^r$) with values in 1-forms. Note that we have $\nabla^{GM}\phi_i=\sum_{j=1}^r\omega_{ji}\wedge\phi_j$ (resp. $\nabla^{\vee GM}\psi_i=\sum_{j=1}^r\omega_{ji}^\vee\wedge\psi_j$). Applying the solution functor to (\ref{Regularization}), applying the commutativity between direct images and solution functors (\cite{Hotta-Takeuchi-Tanisaki} Theorem 7.1.1), and taking a stalk at $y\in U$, we obtain a sequence of isomorphisms
\begin{equation}\label{CohomologyComparison}
\Homo^d(f^{-1}(y)^{an};\mathcal{L}\restriction_{f^{-1}(y)})\simeq(R^df^{an}_*\mathcal{L})_y\simeq (R^df^{an}_!\mathcal{L})_y\simeq\Homo^d_c(f^{-1}(y)^{an};\mathcal{L}\restriction_{f^{-1}(y)}) ,
\end{equation}
where $\mathcal{L}$ is the dual local system of the local system of flat sections $(E^{an})^{\nabla^{an}}$ and $\Homo_c$ stands for the cohomology \mchange{group} with compact support (\cite{Aomoto-Kita} \S 2.2). The first isomorphism of (\ref{CohomologyComparison}) is a result of Lemma \ref{LemmaA} in the appendix. By regularity, we also have the comparison isomorphism of Deligne-Gr\"othendieck (\cite[Corollaire 6.3]{Deligne})
\begin{equation}\label{DeligneGrothendieck}
\Homo^0(f^{-1}(y);\DR_{X/Y}(\D M))\simeq\Homo^d(f^{-1}(y)^{an};\mathcal{L}\restriction_{f^{-1}(y)}).
\end{equation}
Taking the Poincar\'e dual of (\ref{CohomologyComparison}), we obtain a comparison isomorphism of homology groups

\begin{equation}\label{HomologyComparison}
\Homo_d(f^{-1}(y)^{an}; \mathcal{L}\restriction_{f^{-1}(y)})\overset{\sim}{\rightarrow}\Homo_d^{lf}(f^{-1}(y)^{an};\mathcal{L}\restriction_{f^{-1}(y)}).
\end{equation}
Here, we denote by $\Homo^{lf}$ the locally finite (or Borel-Moore) homology group (\cite{Aomoto-Kita} \S 2.2).

Now we define the cohomology intersection matrix. Following \cite{CM}, we will denote it by $I_{ch}$. We denote by $\mathcal{L}^\vee$ the dual local system of $\mathcal{L}$. The cohomology intersection pairing is the perfect pairing between $\Homo^d_c(f^{-1}(y)^{an}; \mathcal{L}\restriction_{f^{-1}(y)})$ and $\Homo^d(f^{-1}(y)^{an}; \mathcal{L}^\vee\restriction_{f^{-1}(y)})$ defined by $\Homo^d(f^{-1}(y)^{an}; \mathcal{L}^\vee\restriction_{f^{-1}(y)})\times \Homo^d_c(f^{-1}(y)^{an}; \mathcal{L}\restriction_{f^{-1}(y)})\ni ([\omega],[\eta])\mapsto \langle[\omega],[\eta]\rangle_{ch}=\int_{f^{-1}(y)}\omega\wedge\eta\in\C$ for a fixed $y$. Note that we take the standard resolution of $\mathcal{L}\restriction_{f^{-1}(y)}$ and $\mathcal{L}^\vee\restriction_{f^{-1}(y)}$ by means of a twisted de Rham complex by regarding $f^{-1}(y)$ as a $2d$-dimensional smooth manifold (\cite[\S 2.2]{Aomoto-Kita}). \mchange{By abuse of notation, we often write $\langle\omega,\eta\rangle_{ch}$ instead of $\langle[\omega],[\eta]\rangle_{ch}$.} In view of isomorphisms (\ref{CohomologyComparison}) and (\ref{DeligneGrothendieck}), we can define the cohomology intersection matrix $I_{ch}=(\langle \phi_i, \psi_j\rangle_{ch})_{i,j}$ which is non-degenerate at each $y\in U$. Since $f^{an}:f^{-1}(U)^{an}\rightarrow U^{an}$ is a fiber bundle, we can take a free basis $\gamma_j\in R^d f^{an}_!(\mathcal{L})$ and $\gamma^\vee_j\in R^d f^{an}_!(\mathcal{L}^\vee)$ on a neighbourhood $W$ of each $y\in U^{an}$. We can define the homology intersection pairing as the perfect pairing $\Homo_d\left(f^{-1}(y)^{an}; \mathcal{L}\restriction_{f^{-1}(y)}\right)\times\Homo^{lf}_d\left(f^{-1}(y)^{an}; \mathcal{L}^\vee\restriction_{f^{-1}(y)}\right)\ni (\gamma,\gamma^\vee)\mapsto \langle\gamma,\gamma^\vee\rangle_{h}\in\C$ which is defined as the Poincar\'e dual of the cohomology intersection pairing (\cite{KY}). By local trivialization, we may assume that the homology intersection matrix $I_h=(\langle \gamma_i,\gamma_j^\vee\rangle_h)_{i,j}$ is constant on $W$, and $I_h$ is non-degenerate in view of (\ref{HomologyComparison}). Now let us put $P=\left(\int_{\gamma_j}\phi_i\right)_{i,j}$ and $P^\vee=\left(\int_{\gamma^\vee_j}\psi_i\right)_{i,j}$. In view of   (\ref{CohomologyComparison}), (\ref{DeligneGrothendieck}) and (\ref{HomologyComparison}), the twisted period relation (\cite{CM} Theorem 2) is 
\begin{equation}\label{TPR}
I_h={}^tP {}^tI_{ch}^{-1}P^\vee.
\end{equation}

Here, ${}^tI_{ch}$ is the transposed matrix of $I_{ch}$. By the definition of the connection matrix, we have two equalities

\begin{equation}
d P={}^t\Omega P,\;\;\;\; d P^\vee={}^t\Omega^\vee P^\vee.
\end{equation}
We put $I=I_{ch}$ and $J={}^tI^{-1}$. Since $I_h$ is locally constant, by differentiating (\ref{TPR}), we obtain the identity

\begin{equation}\label{SecondaryEq2}
dJ+\Omega J+J{}^t\Omega^\vee =0.
\end{equation}
Here, we have used the fact that both $P$ and $P^\vee$ are non-degenerate by the perfectness of the period pairings (\cite{Aomoto-Kita} Lemma 2.5). Taking into account the equality $dI^{-1}=-I^{-1}(dI)I^{-1}$, we have
\begin{equation}\label{SecondaryEq}
dI={}^t\Omega I+I\Omega^\vee.
\end{equation}
We call (\ref{SecondaryEq}) the secondary equation. Note that $\Omega$ and $\Omega^\vee$ are matrices with entries in regular 1-forms on $U$. We also remark that the cohomology intersection matrix $I_{ch}$ is not necessarily a constant matrix for a general choice of the bases $\{ \phi_i\}$ and $\{\psi_i\}$ (see formula (\ref{eqn:4.4}) of this paper). The following theorem is the main result of this paper.

\begin{thm}\label{thm:main}
Suppose that the regularization condition (\ref{Regularization}) is satisfied and $N$ is irreducible. Then, the secondary equation (\ref{SecondaryEq}) is a regular connection, i.e., any analytic solution of (\ref{SecondaryEq}) has at most polynomial growth along any singularity. Moreover, any rational solution of (\ref{SecondaryEq}) is, up to constant multiplication, equal to $I_{ch}$.
\end{thm}

Note that the secondary equation (\ref{SecondaryEq}) can be rewritten in an invariant form: $d\langle \phi,\psi\rangle_{ch}=\langle \nabla^{GM}\phi,\psi\rangle_{ch}+\langle \phi,\nabla^{\vee GM}\psi\rangle_{ch}$. Therefore, we obtain the following corollary.

\begin{cor}
Under the assumption of \cref{thm:main}, let $B$ be an $\mathcal{O}_U$-bilinear form $B:N\otimes_{\mathcal{O}_U}\D N\rightarrow \mathcal{O}_U$ such that for any local sections $\phi$ of $N$ and $\psi$ of $\D N$, the equality
\begin{equation}\label{Compatibility}
dB(\phi,\psi)=B(\nabla^{GM}\phi,\psi)+B(\phi,\nabla^{\vee GM}\psi)
\end{equation}
holds. Then, $B$ is up to constant multiplication equal to the cohomology intersection pairing $\langle\bullet,\bullet\rangle_{ch}$.
\end{cor}

By abuse of notation, let us denote by $U$ a complex smooth quasi-projective variety and by $(E,\nabla_E)$ a connection on $U$. 
In the sequel, we denote this connection simply by $E$ if no confusion arises.
We set
$\mathcal{E}nd(E) = {\cal H}om_{{\cal O}}(E,E)$ and the associated
connection on ${\cal E}nd(E)$ is denoted by $\nabla$.
The endomorphism sheaf ${\cal E}nd(E)$ is again a connection.
In other words, it is a ${\cal D}$-module, is locally free, and of 
finite rank as an ${\cal O}$-module.
See, e.g., \cite[Chapter 5]{Hotta-Takeuchi-Tanisaki}, where one can find 
some fundamental properties of this connection. We recall that analytic continuations of flat sections $(E^{an})^{\nabla_E}$ naturally give rise to a representation of the fundamental group $\pi_1(U^{an},\mathring{x})$ for a base point $\mathring{x}\in U$. We call this representation the monodromy representation. 

\begin{prop}\label{prop:Schur}
Suppose the monodromy representation of $(E,\nabla_E)$ is irreducible. Let $\mathring{x}\in U$ be a point. We put
\begin{equation}
S=\{\varphi\in ({\cal E}nd(E)^{an})^\nabla_{\mathring{x}}\mid \varphi\text{ is monodromy invariant}\}.
\end{equation}
Then, $\dim_{\C} S=1.$
\end{prop}

\begin{proof}
Remember that the connection on ${\cal E}nd(E)$ is given by
\begin{equation}
\langle \nabla\varphi,s\rangle=\nabla_E(\langle\varphi,s\rangle)-\langle\varphi,\nabla_Es\rangle
\end{equation}
for any $\varphi\in{\cal E}nd(E)$ and $s\in E$. Suppose $\varphi\in({\cal E}nd(E)^{an})^\nabla.$ For any $s\in (E^{an})^{\nabla_E}$, we have
\begin{equation}
0=\langle \nabla\varphi,s\rangle=\nabla_E(\langle \varphi,s\rangle).
\end{equation}
We denote by $\pi_E(g)$ (resp. $\pi_{\End}(g)$) the analytic continuation of flat sections of $E$ (resp. ${\cal E}nd(E)$) along a loop $g\in\pi_1(U^{an},\mathring{x})$. We have
\begin{equation}
\pi_E(g)(\langle \varphi,s\rangle)=\langle \pi_{\End}(g)\varphi,\pi_E(g)s\rangle\mchange{,}
\end{equation}
\mchange{and therefore} 
\begin{equation}
\pi_{\End}(g)\varphi=\pi_E(g)\circ\varphi\circ\pi_E(g^{-1}).
\end{equation}
From this, we can see that $\pi_{\End}(g)\varphi=\varphi$ is equivalent to $\pi_E(g)\circ\varphi=\varphi\circ\pi_E(g)$. Therefore, we have an identity
\begin{equation}
S=\End_{\pi_1(U^{an},\mathring{x})}(E_{\mathring{x}}).
\end{equation}
The assertion follows from Schur's lemma.
\end{proof}

Next, we recall the trivialisation formula for a tensor connection. Let $(F,\nabla_F)$ be another connection on $U$. Suppose that $(E,\nabla_E)$ and $(F,\nabla_F)$ are trivialized with respect to frames $({\bf e}_1,\dots,{\bf e}_r)$ and $({\bf f}_1,\dots,{\bf f}_{r^\prime})$. The connections are trivialized as $\nabla_E=d+A\wedge$ and $\nabla_F=d+B\wedge$, where $A=(\omega^i_j)$ and $B=(\tilde{\omega}^i_j)$ are square matrices with entries in $1$ forms. Then, for any section $\sum_{\substack{1\leq i\leq r\\1\leq j\leq r^\prime}}\alpha^{ij}{\bf e}_i\otimes{\bf f}_j$ of $E\otimes F := E \otimes_{{\cal O}_U} F$, we have

\begin{align}
\nabla_{E\otimes F}\left(\sum_{\substack{1\leq i\leq r\\1\leq j\leq r^\prime}}\alpha^{ij}{\bf e}_i\otimes{\bf f}_j\right)&=\sum_{i,j}d\alpha^{i,j}{\bf e}_i\otimes{\bf f}_j+\sum_{j}({\bf e}_1,\dots,{\bf e}_r)A
\left(
\begin{array}{ccc}
\alpha^{1j}\\
\vdots\\
\alpha^{rj}
\end{array}
\right)\otimes{\bf f}_j\nonumber
\\
&+
\sum_{i}{\bf e}_i\otimes
({\bf f}_1,\dots,{\bf f}_{r^\prime})
B
\left(
\begin{array}{ccc}
\alpha^{i1}\\
\vdots\\
\alpha^{ir^\prime}
\end{array}
\right)
\\
 &=\sum_{i,j}d\alpha^{i,j}{\bf e}_i\otimes{\bf f}_j+\sum_{i,j}(\sum_{k=1}^r\omega^i_k\alpha^{kj}){\bf e}_i\otimes{\bf f}_j+\sum_{i,j}(\sum_{l=1}^{r^\prime}\tilde{\omega}^j_l\alpha^{il}){\bf e}_i\otimes{\bf f}_j.
\end{align}

Therefore, if we trivialize the tensor product $E\otimes F$ with respect to the frame $\{ {\bf e}_i\otimes{\bf f}_j\}$, the connection is given by $\nabla_{E\otimes F}=d+A\bullet+\bullet{}^tB$. Now we can show that the secondary equation is actually a tensor connection. 

\begin{prop}\label{prop:Tensor}
The secondary equation (\ref{SecondaryEq}) coincides with the tensor connection $\D N\otimes N$.
\end{prop}

\begin{proof}
Recall that the connection of $N$ is given by $d+\Omega$ with respect to the basis $\phi_j$ and that of $\D N$ is given by $d+\Omega^\vee$ with respect to the basis $\psi_j$. Thus, their dual connections with respect to the dual frames are given by $d-{}^t\Omega$ (connection for $\D N$) and $d-{}^t\Omega^\vee$ (connection for $N$) respectively. Therefore, the connection $ \D N\otimes N$ with respect to these frames is given by
$d-{}^t\Omega\bullet-\bullet\Omega^\vee$. This is nothing but the secondary equation (\ref{SecondaryEq}). 
\end{proof}

\begin{lem}\label{lem:Meromorphy}
If $(E,\nabla_E)$ is a regular connection, and if $s\in \Gamma(U^{an},(E^{an})^{\nabla_E})$ is monodromy invariant, we have $s\in\Gamma(U,E)$, i.e.\ , $s$ is an algebraic section.
\end{lem}

\begin{proof}
Take a projective compactification $X$ of $U$ so that $D=X\setminus U$ is a normal crossing divisor. Since $s$ has at most polynomial growth along $D$ (\cite{Deligne} Th\'eor\`eme 4.1), we see that there is a positive integer $m$ such that $s\in\Gamma(X^{an},\mathcal{O}_E^{an}(mD))=\Gamma(X,\mathcal{O}_E(mD))$. Here, the equality is a consequence of GAGA (\cite{Serre}).
\end{proof}

\noindent
({\rm Proof of \cref{thm:main}}) \mchange{Recall} that $N$ is a regular connection, hence so is $\D N$. Therefore the tensor connection $\D N\otimes N$ is also regular. By \cref{prop:Tensor}, we can conclude that (\ref{SecondaryEq}) is regular. The first part of the  statement is verified.

In view of \cref{prop:Schur}, \cref{lem:Meromorphy}, and the fact that $\D N\otimes N$ is isomorphic to ${\cal E}nd(N)$, we see that the rational solutions of (\ref{SecondaryEq}) are one dimensional. Taking into account that $I_{ch}$ is monodromy invariant, the second part of the statement is confirmed.

\section{Euler integral representations and twisted period relations}
In this section, we discuss general results on Euler integral representations. Consider $k$ Laurent polynomials
$h_{l,z^{(l)}}(x)=\displaystyle\sum_{j=1}^{N_l}z_j^{(l)}x^{{\bf a}^{(l)}(j)}\;\;(l=1,\dots,k),$
where each coefficient $z_j^{(l)}$ is regarded as a variable and $x=(x_1,\dots,x_n)$. For any parameters $\gamma_l\in\C\;\;(l=1,\dots k)$ and $c\in\C^{n},$ an integral 
\begin{equation}\label{EulerInt}
f_{\Gamma}(z)=\int_\Gamma h_{1,z^{(1)}}(x)^{-\gamma_1}\cdots h_{k,z^{(k)}}(x)^{-\gamma_k}x^{c} \frac{dx}{x}
\end{equation}
is called Euler integral. In the formula above, we put $\frac{dx}{x}=\frac{dx_1}{x_1}\wedge\dots\wedge\frac{dx_n}{x_n}.$ Here, $\Gamma$ is a suitable element of the twisted homology group associated to the multivalued function 
\begin{equation}\label{Phi}
\Phi=h_{1,z^{(1)}}(x)^{-\gamma_1}\cdots h_{k,z^{(k)}}(x)^{-\gamma_k}x^{c}.
\end{equation}
Let us clarify the meaning of this choice. We define an algebraic connection on a trivial bundle over $(\Gm)^n\setminus\{ x\in(\Gm)^n\mid h_{1,z^{(1)}}(x)\cdots h_{k,z^{(k)}}(x)=0\}$ by 
\begin{equation} 
\nabla_x=d_x-\sum_{l=1}^k\gamma_l\frac{d_xh_{l,z^{(l)}}(x)}{h_{l,z^{(l)}}(x)}\wedge+\sum_{i=1}^nc_i\frac{dx_i}{x_i}\wedge.
\end{equation}
Here, $d_x$ is the exterior derivative in $x$-variables. Formally, the action of the connection $\nabla_x$ on any function $f$ is given by the formula $\nabla_xf=\Phi^{-1}\times(d_x(\Phi \times f)).$ We denote by $\C\Phi$ the dual local system of the local system of flat sections of the analytification $\nabla^{an}_x$. Then, our integration cycle $\Gamma$ belongs to the twisted homology group
\begin{equation}
\Homo_n\left( (\C^\times)^n\setminus\{ x\in(\C^\times)^n\mid h_{1,z^{(1)}}(x)\cdots h_{k,z^{(k)}}(x)=0\}; \C \Phi\right).
\end{equation}

To control the Euler integral $f_{\Gamma}$, we use a GKZ system via the Cayley trick. We put $A_l=({\bf a}^{(l)}(1)|\dots|{\bf a}^{(l)}(N_l))$, $N=N_1+\dots+N_k$ and define an $(n+k)\times N$ matrix $A$ by

\begin{equation}\label{CayleyConfigu}
A
=
\left(
\begin{array}{ccc|ccc|c|ccc}
1&\cdots&1&0&\cdots&0&\cdots&0&\cdots&0\\
\hline
0&\cdots&0&1&\cdots&1&\cdots&0&\cdots&0\\
\hline
 &\vdots& & &\vdots& &\ddots& &\vdots& \\
\hline
0&\cdots&0&0&\cdots&0&\cdots&1&\cdots&1\\
\hline
 &A_1& & &A_2& &\cdots & &A_k& 
\end{array}
\right).
\end{equation}

\noindent
We put $\delta=\begin{pmatrix}\gamma\\ c\end{pmatrix}$ and denote by $M_A(\delta)$ the GKZ system (\cite{GKZEuler}). Note that our parameters $\gamma$ and $c$ correspond to $-\alpha$ and $\beta$ in \cite{GKZEuler}. The following proposition is well-known.
\begin{prop}[\cite{GKZEuler}]
The integral (\ref{EulerInt}) is a solution of $M_A(\delta).$
\end{prop}

\noindent
More generally, under a suitable condition on parameters $\delta$, one can prove that any solution of $M_A(\delta)$ has an Euler integral representation (\cite{GKZEuler}). 

Let $\pi:X=(\Gm)_x^n\times \A^N_z\setminus\bigcup_{l=1}^k\{ (x,z)\mid h_{l,z^{(l)}}(x)=0\}\rightarrow \A^N_z=Y$ be \mchange{the} natural projection where subscripts stand for coordinates. We define an algebraic connection of \mchange{the} trivial bundle over $X$ by $\nabla=\Phi^{-1}\circ d_{(x,z)}\circ\Phi.$ We denote by $\Z A$ the sub-lattice of $\Z^{(n+k)\times 1}$ spanned by column vectors of $A$. If $\Z A=\Z^{(n+k)\times 1}$, $\delta$ is non-resonant in the sense of \cite{GKZEuler}, and $\gamma_l\notin\Z$ for any $l=1,\dots,k$, we have the regularization isomorphism $\int_{\pi!}(\mathcal{O}_X,\nabla)\overset{\sim}{\rightarrow}\int_{\pi}(\mathcal{O}_X,\nabla)$ as well as a canonical isomorphism $M_A(\delta)\simeq\int_{\pi}(\mathcal{O}_X,\nabla)$ by \cite{MH} Theorem 2.12. \mchange{Note that the regularization condition is in general violated if any of $\gamma_l$ is an integer (\cite[Remark 2.14]{MH}).} If we denote by $U$ the Newton non-degenerate locus of $M_A(\delta)$, $M_A(\delta)$ is a connection on $U$ (\cite{Adolphson} LEMMA 3.3).

Now let us fix a basis $\{\phi_i(z)\}_{i=1}^r\subset\int^0_{\pi}(\mathcal{O}_X,\nabla)\restriction_{U}$ and a basis $\{\psi_i(z)\}_{i=1}^r\subset\int^0_{\pi}(\mathcal{O}_X,\nabla^{\vee})\restriction_{U}$. By using the relative de Rham complex, we can explicitly compute the connection $\nabla^{GM}$ of $\int^0_{\pi}(\mathcal{O}_X,\nabla)\restriction_{U}$ via the formula 
\begin{equation}\label{GMAction}
\nabla^{GM}\phi=d_z\phi-\displaystyle\sum_{j,l}\gamma_l\frac{x^{{\bf a}^{(l)}(j)}}{h_{l,z^{(l)}}(x)}dz_j^{(l)}\wedge\phi
\end{equation}
for any $\phi\in\Homo^n\left( (\Gm)^n_x\times U;\left(\Omega^\bullet_{(\Gm)^n_x\times U/U}\left(*\cup_{l=1}^k\{ h_{l,z^{(l)}}(x)=0\}\right),\nabla_x\right)\right)=\int^0_{\pi}(\mathcal{O}_X,\nabla)\restriction_{U}$. We have a similar formula for the connection $\nabla^{\vee GM}$ of $\int^0_{\pi}(\mathcal{O}_X,\nabla^{\vee})\restriction_{U}$. These bases as well as the connection matrices $\nabla^{GM}=d_z+\Omega$ and $\nabla^{\vee GM}=d_z+\Omega^\vee$ can explicitly be constructed via Gr\"obner bases (\cite{HibiTakayamaNishiyama} Theorem 2). Note that there is an isomorphism induced by the correspondence  $M_A(\delta)\ni [1]\mapsto [\frac{dx}{x}]\in\int_{\pi}^0(\mathcal{O}_X,\nabla)$ (\cite[Corollary 3.4]{MH}).
When $A$ is a matrix of the form (\ref{CayleyConfigu}), the associated GKZ system is
regular holonomic \cite{Hotta-Wuhan}.
In \cite{MH}, the intersection numbers for a basis of the twisted homology
group for the GKZ system associated to a matrix $A$ of the form (\ref{CayleyConfigu}) admitting
a unimodular regular triangulation
are determined
when parameters are generic (\cite{MH} Theorem 8.1).
The irreducibility of the GKZ system is proved under the non-resonance 
condition \cite{B2010}, \cite{GKZEuler}, \cite{SW2010}.
Note that irreducibility of a regular holonomic $\mathcal{D}$-module is equivalent to that of the corresponding perverse sheaf by Riemann-Hilbert correspondence (\cite{Hotta-Takeuchi-Tanisaki} Theorem 7.2.5).
Algorithms for finding
rational solutions have been studied by several approaches 
(see, e.g., \cite{Oaku-Takayama-Tsai} and its references).
These results together with our Theorem \ref{thm:main} yield the following
theorem.

\begin{thm}  \label{th:algorithm}
Given 
a matrix $A=(a_{ij})$ as in (\ref{CayleyConfigu})
admitting a unimodular regular triangulation $T$.
When parameters are non-resonant, $\gamma_l\notin\Z$ and moreover the set of series solutions
\mchange{with respect to} $T$ is linearly independent, 
the intersection matrix of
the twisted cohomology group of the GKZ system associated to the matrix $A$
can be algorithmically determined.
\end{thm}

We \mchange{rename $(z_j^{(l)})_{j,l}$ as $(z_i)_i$ and} denote $\frac{\partial}{\partial z_i}$ by $\pd{i}$.
The action of a differential operator to a function is denoted by $\bullet$.
In other words, $\pd{i}\bullet f$ means $\frac{\partial f}{\partial z_i}$.
Moreover, we denote by $\Omega_i$ the coefficient matrix of $\Omega$ with respect to the $1$-form $dz_i$.
The algorithm we propose is summarized as follows.
\begin{enumerate}
\item Obtain a Pfaffian system $\pd{i} + \Omega_i$
from the GKZ hypergeometric ideal \mychange{defining} $M_A(\delta)$ generated by
\begin{eqnarray*}
&&  \sum_j a_{ij} z_j \pd{j} +\delta_i, \quad i=1, \ldots, n+k \\
&&  \pd{}^u - \pd{}^v,  \quad Au=Av, u,v \in {\bf N}_0^N
\end{eqnarray*}
It is well-known that 
this step can be performed by a Gr\"obner basis computation 
in the ring of differtial operators with rational function coefficients.
See, e.g., \cite[6.2, 7.4.2]{dojo-en}.
\item Find a non-zero rational function solution $I$ of the secondary equation
$$ \pd{i}\bullet I - {}^t\Omega_i I - I \Omega_i^\vee =0,
\quad i=1, \ldots, N. $$
To be more precise, see, e.g.,  \cite{Oaku-Takayama-Tsai} and its references.
\item Determine the constant multiple of $I$ by the twisted period relation (\ref{TPR}). To be more precise, we use \cite[Theorem 8.1]{MH}.
\end{enumerate}
Let $G$ be the Gr\"obner basis of the GKZ hypergeometric ideal
obtained in step 1.
The set of the standard
monomials for $G$,
which is the set of monomials in $\pd{}$ that are not divisible
by \mychange{the initials of} the elements of $G$,
is of the form
$\{ \pd{}^s \,|\, s \in S \subset {\bf N}_0^N\}$  
(see, e.g., \cite[6.1]{dojo-en}). 
The basis of the twisted cohomology groups corresponding to $G$
is obtained by applying these $\pd{}^s$'s to 
the kernel function of (\ref{EulerInt}) 
and dividing it by (\ref{Phi}). 

\begin{exa}\label{example:gauss} \rm  
Let us determine algorithmically the intersection matrix of the twisted cohomology group
for a matrix standing for the hypergeometric function ${}_2F_1$,
which was studied by \cite{CM} and \cite{Matsumoto-osaka} 
with geometric methods.
We put
$$
A=\left(
\begin{array}{cc|cc}
1 & 1 & 0 & 0 \\ \hline
0 & 0 & 1 & 1 \\ \hline
0 & 1 & 0 & 1
\end{array}
\right)
, \quad
F=\left(\begin{array}{c} 1 \\ z_4 \partial_4 \end{array}\right).
$$
Here, $\{1, \pd{4}\}$ in the frame $F$ is the set of the standard monomials
of a Gr\"obner basis of \mychange{the left ideal defining} $M_A(\delta)$.
The bases of the twisted cohomology group corresponding to the frame $F$ is
$\left\{ \frac{dx}{x}, 
 z_4 \frac{\partial \log\mchange{\Phi}}{\partial z_4} \frac{dx}{x}=
 \frac{-\gamma_2 z_4 dx}{(z_3+z_4 x)x} \right\}\subset \int^0_{\pi}(\mathcal{O}_X,\nabla)\restriction_{U}$ and $\left\{ \frac{dx}{x}, 
 z_4 \frac{\partial \log \mchange{\Phi}^{-1}}{\partial z_4} \frac{dx}{x}=
 \frac{\gamma_2 z_4 dx}{(z_3+z_4 x)x} \right\}\subset \int^0_{\pi}(\mathcal{O}_X,\nabla^{\vee})\restriction_{U}$
where $\mchange{\Phi} = (z_1+z_2 x)^{-\gamma_1} (z_3 + z_4 x)^{-\gamma_2} x^c$.
\rrr
The expression 
$\mchange{\nabla^{GM}_{\pd{z_4}}} \left[ \frac{dx}{x} \right]
= \frac{-\gamma_2 z_4 dx}{(z_3+z_4 x)x}$
is obtained by the exchange of $\pd{z_4}$ and the integral sign as
\begin{eqnarray*}
& & \pd{z_4} \bullet \int_\Gamma (z_1+x z_2)^{-\gamma_1} (z_3 + x z_4)^{-\gamma_2} x^c
   \cdot \frac{dx}{x} \\
&=& \int_\Gamma (z_1+x z_2)^{-\gamma_1} (z_3 + x z_4)^{-\gamma_2} x^c \cdot
   \frac{-\gamma_2 z_4}{z_3+z_4 x}
   \frac{dx}{x} 
\end{eqnarray*}
and the perfectness of the twisted homology and cohomology groups.
The other expressions can be obtained analogously.
Generally speaking, through the connection $\nabla^{GM}$ in (\ref{GMAction}),
we can define an action of the Weyl algebra $D_N$ on
$\int_\pi^0 (\mathcal{O}_X,\nabla)$;
the action of $\pd{z_i}$ is given by $\nabla^{GM}_{\pd{z_i}}$.
When ${\cal S}$ is the set of the standard monomials
of a Gr\"obner basis of the left ideal defining $M_A(\delta)$,
$\left\{ s \bullet \left[ \left. \frac{dx}{x} \right] \,\right|\, s \in {\cal S} \right\}$
is a free basis of ${\cal O}_{U'}$-module $\int_\pi^0 (\mathcal{O}_X,\nabla)|_{U'}
\simeq M_A(\delta)|_{U'}$
where $U'$ is the complement of the singular locus of $\Omega_i$'s
in Theorem \ref{th:algorithm} associated to the Gr\"obner basis.

Now, we
\bbb
set $z_1=z_2=z_3=1$ because of two reasons.
The first reason is that the solutions of the GKZ system can be expressed as a composition
of rational functions and transcendental functions of one variable
by the homogeneity condition (first order operators in $M_A(\delta)$);
see, e.g., \cite[Prop 1.3,7]{SST}).
The second reason is that if we do not set these $z_i$'s to $1$,
the procedure to get a rational solution requires more computational resources,
which are wasteful because of the homogeneity condition.
For the frame $F$, we have the following connection and dual connection matrices
$$
\Omega=\left(
\begin{array}{cc}
0& \frac{    - c\gamma_2} {   {z}_{4}- 1}  \\
 \frac{ 1} {  {z}_{4}}&  \frac{     -(  c+ {\gamma}_{2})  {z}_{4}+c- {\gamma}_{1}} {   {z}_{4}  (  {z}_{4}- 1)} \\
\end{array}
\right) dz_4
$$
$$
\Omega^\vee
=\left(
\begin{array}{cc}
0&  \frac{    - c\gamma_2} {   {z}_{4}- 1}\\
 \frac{ 1} {  {z}_{4}} &  \frac{     (  c+ {\gamma}_{2})  {z}_{4}-c+ {\gamma}_{1}} {   {z}_{4}  (  {z}_{4}- 1)}  \\
\end{array}
\right) dz_4
$$
These can be obtained by a Gr\"obner basis of the hypergeometric ideal $M_A(\delta)$
(see, e.g., \cite[Chapter 6]{dojo-en}, \cite{SST}).
We find the following rational solution of the secondary equation
$ d_{z_4} I - {}^t \Omega I - I \Omega^\vee=0$. 
$$
I=
\left(
\begin{array}{cc}
 1&    \frac{ c\gamma_2} { \gamma_1+\gamma_2} \\
 \frac{ -c\gamma_2} { \gamma_1+\gamma_2}&  \frac{ c\gamma_2 ( {\gamma}_{1}- c)} {\gamma_1+\gamma_2 } 
\end{array}
\right).
$$
Let us utilize the quadratic relation 
of \cite[Theorem 8.1]{MH}
to determine the constant multiple of $I$.
Let $T$ be a unimodular triangulation $\{123,234\}$.
From the relation, we have 
$$\gamma_1 \gamma_2  \sum_{\sigma \in T} \frac{\pi^{3}}{\sin \pi A_\sigma^{-1} \delta}
 \varphi_{\sigma,0}(z; \delta) \varphi_{\sigma,0}(z; -\delta) =
\frac{\langle \frac{dx}{x}, \frac{dx}{x} \rangle_{ch}}{ 2 \pi \sqrt{-1}}
$$
where $\varphi_{\sigma,0}(z;\delta)$ is the series solution of the form
$$
\sum_{k \in {\bf N}_0} \frac{z^{kL+\rho}}{\prod_{i=1}^4 \Gamma(1+\rho_i+ k L_i)} ,
\quad
 L=(1,-1,-1,1), 
 A \rho + \delta =0, \rho_i = 0 \ \mbox{when $i\not\in \sigma$}
$$
and $A_\sigma$ is the submatrix of $A$ constructed by taking the columns standing for
$\sigma$.
For example, when $\sigma=234$, $A_\sigma$ is the $3 \times 3$ matrix of which columns
are the second, third, and the 4th columns of $A$.
The constant term of the left-hand side is 
\begin{equation} \label{eq:dxdx20181228}
\frac{1}{\gamma_1-c} \left( \frac{\gamma_1}{c} + \frac{\gamma_2}{c-\gamma_1-\gamma_2} \right)
= \frac{   -  (  \gamma_{1}+ \gamma_{2})} {   {c}  (   {c}- \gamma_{1}- \gamma_{2})}
\end{equation}
by utilizing the rule $\Gamma(1+z)\Gamma(1-z)=z \frac{\pi}{\sin \pi z}$. We multiply a constant to $I$ so that the $(1,1)$ element 
is equal to the constant (\ref{eq:dxdx20181228}) $\times 2 \pi \sqrt{-1}$.
Thus, we obtain the following intersection matrix of the twisted cohomology group: 
\begin{equation}
2 \pi \sqrt{-1} 
\left( \begin{array}{cc}
 \frac{   -   (  {\gamma}_{1}+ {\gamma}_{2})} {   {c}  (   {c}- {\gamma}_{1}- {\gamma}_{2})}&  \frac{   -  {\gamma}_{2}} {    {c}- {\gamma}_{1}- {\gamma}_{2}} \\
 \frac{  {\gamma}_{2}} {    {c}- {\gamma}_{1}- {\gamma}_{2}}&  \frac{   {\gamma}_{2}  (  {c}- {\gamma}_{1})} { {c}- {\gamma}_{1}- {\gamma}_{2}} \\
\end{array} \right).
\end{equation}
\end{exa}

\bigskip

We conclude this section with simple observations on the cohomology intersection numbers associated to GKZ system. The first observation is on the dual fan of the Newton polytope of the LCM of the denominators (common denominator) of the entries of $I_{ch}$. To begin with, we recall the basic notion of Newton polytope (\cite[Chapter 6]{GKZbook}). 
Let $f(z)=\sum_\alpha f_\alpha z^\alpha$ be a polynomial in $z=(z_1,\dots,z_N)$. The Newton polytope ${\rm New}(f)$ is the convex hull of the index set $\{\alpha\mid f_\alpha\neq 0\}$ in $\R^{N\times 1}$. For any convex polytope $P$ in $\R^{N\times 1}$ and its face $F$, its normal cone $N_F(P)$ is defined to be the set of $\omega\in(\R^{N\times 1})^*=\R^{1\times N}$ such that the face determined by the vector $-\omega$ is $F$ ($\langle -w, x\rangle$ takes the maximum on $F$ while $x\in P$). The dual fan of the  Newton polytope ${\rm New}(f)$ is denoted by ${\rm N}(f)$. Now let us denote by $E_A(z)$ the product of principal $A_{\Gamma}$-discriminants for any face $\Gamma$ of ${\rm New}(A)$ (\cite{GKZbook} Chapter 9 Definition 1.2). Here, ${\rm New}(A)$ is the convex hull of column vectors of $A$. Since the secondary equation is a tensor product of $M_A(\delta)$ and $M_A(-\delta)$, its singular locus is contained in the vanishing locus of $E_A(z)$. If $g(z)$ denotes the common denominator of $I_{ch}$, we have $\{ g(z)=0\}\subset\{ E_A(z)=0\}$. By Hilbert's nullstellensatz, there is a polynomial $h(z)$ and a positive integer $l$ such that $E_A(z)^l=g(z)h(z)$. Taking their Newton polytopes, we have ${\rm New}(E_A^l)={\rm New}(g)+{\rm New}(h)$. Here $+$ is the Minkowski sum of convex polytopes. If we take their dual fans, we see that the left hand side gives the secondary fan (\cite[Chapter 7, Proposition 1.5, Chapter 10, Theorem 1.4]{GKZbook}) while the right hand side is a fan which is a refinement of ${\rm N}(g)$. Summing up all the arguments above, we obtain a
\begin{thm}
Suppose that the parameter $\delta$ is non-resonant and $\gamma_l\notin\Z$ for any $l=1,\dots,k$. Let $g$ be the common denominator of $I_{ch}$. Then, the secondary fan  is a refinement of ${\rm N}(g)$.
\end{thm}

\noindent
Note that by \cite[Chapter 6, Corollary 1.6.]{GKZbook}, there is an injective correspondence from open cones of ${\rm N}(g)$ to convergence domains of Laurent series expansions of $\frac{1}{g(z)}$. 

The second observation is on how the cohomology intersection number depends on the parameters $\delta$. In Example \ref{example:gauss}, the basis $F$ depends rationally on the parameters $\delta$. In this case, the cohomology intersection numbers are rational functions in both the variables $z$ and the parameters $\delta$ with coefficients in the field of rational numbers $\mathbb{Q}$. Let us formulate the observation above. For any field extension $K\subset\C$ of $\Q$, we set $(\Gm(K))^n_x={\rm Spec}\left( K[x_1^\pm,\dots,x_n^\pm]\right)$. Since the (reduced) defining equation $E_A(z)$ of the complement $Z=\A^N\setminus U$ is a rational polynomial, we can also consider a reduced scheme $U(K)$ defined over $K$ whose base change to $\C$ is isomorphic to $U$. We denote by $D_{U(K)}$ the ring of differential operators on $U(K)$. Any element $P$ of $D_{U(K)}$ is a finite sum $P=\frac{1}{E_A(z)^l}\sum_\alpha a_\alpha(z)\partial^\alpha$ where $a_\alpha(z)$ is a polynomial with coefficients in $K$ and $l$ is an integer.

\begin{thm}\label{thm:Coefficients}
Suppose that $\delta$ is non-resonant, $\gamma_l\notin\Z$, and $A$ as in (\ref{CayleyConfigu}) admits a unimodular regular triangulation $T$. Then, for any $P_1,P_2\in D_{U(\mathbb{Q}(\delta))}$, the cohomology intersection number $\frac{\langle P_1\cdot\frac{dx}{x},P_2\cdot\frac{dx}{x}\rangle_{ch}}{(2\pi\ii)^n}$ belongs to the field $\mathbb{Q}(\delta)(z)$.
\end{thm}

\begin{proof}
By the definition of the action of the ring $\mathbb{Q}(\delta)\langle z,\partial_z\rangle$, it is enough to show that the function $I(z;\delta)=\frac{\langle x^{\bf a}h^{\bf b}\frac{dx}{x},x^{{\bf a}^\prime}h^{{\bf b}^\prime}\frac{dx}{x}\rangle_{ch}}{(2\pi\ii)^n}$ belongs to the field $\mathbb{Q}(\delta)(z)$. Here, ${\bf a}=(a_1,\dots,a_n),{\bf a}^\prime\in\Z^n$, ${\bf b}=(b_1,\dots,b_k),{\bf b}^\prime\in\Z^k$ and $x^{\bf a}h^{\bf b}=x_1^{a_1}\cdots x_n^{a_n}h_1^{b_1}\cdots h_k^{b_k}$. Due to \cite[Lemma 3.2]{AET}, we may assume that the parameter $\delta$ is generic so that the formula \cite[Theorem 8.1]{MH} holds. Let $B$ be an $N\times (N-n-k)$ integer matrix whose column vectors generate the kernel lattice $\Ker(A\times:\Z^{N}\rightarrow \Z^{n+k})$ and are compatible with the regular triangulation $T$ in the sense of \cite{GKZToral}. We take an $(N-n-k)\times N$ integer matrix $C$ such that $CB$ is an identity matrix. We define a morphism $j_A:(\C^*)^{n+k}\rightarrow (\C^*)^N$ by $j_A(t)=t^A$ and a morphism $\pi_B:(\C^*)^N\rightarrow(\C^*)^{N-n-k}$ by $\pi_B(z)=z^B$. Note that the morphism $j_A$ combined with the product structure of $(\C^*)^N$ induces an action of $(\C^*)^{n+k}$ on $(\C^*)^N$. For any $t=(t_1,\dots,t_{n+k})$ in $(\C^*)^{n+k}$, we write $\tau_1=(t_1,\dots,t_k)$ and $\tau_2=(t_{k+1},\dots,t_{n+k})$. By \cite[Theorem 8.1]{MH}, we have a formula $I(j_A(t)z;\delta)=\tau_1^{\bf b+b^\prime}\tau_2^{\bf -a-a^\prime} I(z;\delta).$ Taking into account this homogeneity property and the exact sequence
\begin{equation}
1\rightarrow(\C^*)^{n+k}\overset{j_A}{\rightarrow}(\C^*)^{N}\overset{\pi_B}{\rightarrow}(\C^*)^{N-n-k}\rightarrow 1,
\end{equation}
we only have to show that $I_1(\zeta;\delta)\overset{def}{=}I(\zeta^C;\delta)\in\mathbb{Q}(\delta)(z)$. Here, $\zeta$ is a coordinate of $(\C^*)^{N-n-k}$. Note that the Laurent series expansion formula \cite[Theorem 8.1]{MH} is valid when the absolute values of entries of $\zeta$ are all small. Since $I(z;\delta)$ is a solution of the secondary equation (\ref{SecondaryEq}) and (\ref{SecondaryEq}) is isomorphic to $M_A(\delta)\otimes M_A(-\delta)$, for an appropriate non-negative integer $l$ and a polynomial $f(\zeta;\delta)$ in $\zeta$ and a vector $\alpha_0\in\mchange{{\bf N}_0}^{N-n-k}$, we have $I_1(\zeta;\delta)=\frac{f(\zeta;\delta)}{\zeta^{\alpha_0} E_A(\zeta^C)^l}$. Note that $E_A(z)\in\mathbb{Q}[z]$. We write $f(\zeta;\delta)=\sum_{i=1}^df_i(\delta)\zeta^{\alpha_i}$ where $\alpha_i$ are elements of $\mchange{{\bf N}_0}^{N-n-k}$. By choosing a suitable vector $\phi=(\phi_1,\dots,\phi_{N-n-k})\in\mchange{{\bf N}_0}^{N-n-k}$ and by rearranging $\alpha_i$ if necessary, we may assume $0\leq\langle \phi,\alpha_1\rangle<\cdots<\langle \phi,\alpha_d\rangle$. Here, $\langle\bullet,\bullet\rangle$ is the dot product of vectors. We set $g(\zeta)=E_A(\zeta^C)^l$. For a complex variable $\xi$, we write $\xi^\phi=(\xi^{\phi_1},\dots,\xi^{\phi_{N-n-k}}).$ We put $F(\xi)\overset{def}{=}f(\xi^\phi;\delta)=\xi^{\langle \phi,\alpha_0\rangle}g(\xi^{\phi})I_1(\xi^\phi;\delta).$ By \cite[Theorem 8.1]{MH} and the fact that $B$ is compatible with $T$, we see that $I_1(\xi^\phi;\delta)$ is a Laurent series in $\xi$ with coefficients in $\mathbb{Q}(\delta)$. Since $F(\xi)=\sum_{i=1}^df_i(\delta)\xi^{\langle \phi,\alpha_i\rangle}$, we obtain $f_i(\delta)\in\mathbb{Q}(\delta)$. Thus, we have $f(\zeta;\delta)\in\mathbb{Q}(\delta)[\zeta]$, hence we \mchange{obtain} $I_1(\zeta;\delta)\in\mathbb{Q}(\delta)(z)$.

\end{proof}

For any field extension $K\subset\C$ of $\Q$, we define the symbol $\int^0_{\pi}(\mathcal{O}_{X(K)},\nabla)\restriction_{U(K)}$ by
\begin{equation}
\mchange{\mathbb{H}}^n\left( (\Gm(K))^n_x\times U(K);\left(\Omega^\bullet_{(\Gm(K))^n_x\times U(K)/U(K)}\left(*\bigcup_{l=1}^k\{ h_{l,z^{(l)}}(x)=0\}\right),\nabla_x\right)\right),
\end{equation}
\mchange{where $\mathbb{H}^n$ stands for the $n$-th hypercohomology group.} By the formula (\ref{GMAction}), $\int^0_{\pi}(\mathcal{O}_{X(K)},\nabla)\restriction_{U(K)}$ naturally has a structure of $D_{U(K)}$-module. When $\delta$ is non-resonant, $\gamma_l\notin\Z$, and $K=\Q(\delta)$, any element $[\phi]\in\int_{\pi}^0(\mathcal{O}_{X(K)},\nabla)\restriction_{U(K)}$ can be written as $[\phi]=P\cdot[\frac{dx}{x}]$ for some $P\in D_{K}$. Therefore, we obtain a

\begin{thm}\label{thm:Coefficients2}
Under the assumption of \cref{thm:Coefficients}, the normalized cohomology intersection pairing $B=\frac{\langle\bullet,\bullet\rangle_{ch}}{(2\pi\ii)^n}$ defines a perfect bilinear pairing 
\begin{equation}
\int^0_{\pi}(\mathcal{O}_{X(\Q(\delta))},\nabla)\restriction_{U(\Q(\delta))}\times\int^0_{\pi}(\mathcal{O}_{X(\Q(\delta))},\nabla^\vee)\restriction_{U(\Q(\delta))}\rightarrow\mathcal{O}_{U(\Q(\delta))}
\end{equation}
such that the formula (\ref{Compatibility}) holds.
\end{thm}

\noindent
Therefore, when the parameter $\delta$ is generic, we can treat it as a formal symbol and we do not need to consider any algebraic extension of $\Q$ in our algorithm.

\section{A period integral associated to a family of K3 surfaces (\cite{NS01})}

This section is a demonstration of our algorithmic method
to obtain the intersection  matrix of twisted cohomology groups
and functional identities derived by the twisted period relation.
The intersection matrix and some functional identities
presented in this section have not been obtained by other methods.
We also note that once we obtain the intersection matrix, the correctness
modulo a constant multiple can be checked if it satisfies the secondary
equation (\ref{SecondaryEq}).  Since the entries of the matrix are rational functions,
it can be checked by computer algebra systems.

We consider an Euler integral 
\begin{equation}
f(z)=\int_\Gamma u\frac{dx\wedge dy}{xy},\ 
u=(z_1x^3+z_2x^2y+z_3x^2y^{-1}+z_4x^2+z_5x)^{-c_1}x^{c_2}y^{c_3}
\end{equation}
under the non-resonance condition on $c$.
This function is a solution of the GKZ system
associated to the matrix
$$
\left(\begin{array}{c}1\\{\bf a}^{(1)}(j)\end{array}\right) = \left(
\begin{array}{ccccc}
  1 & 1 & 1 & 1 & 1 \\
  3 & 2 & 2 & 2 & 1 \\
  0 & 1 & -1 & 0 & 0 \\
\end{array}
\right).
$$
We can see that the sets of rational differential forms
\begin{equation}
\omega=\frac{dx \wedge dy}{xy},
 \frac{\partial \log u}{\partial z_5} \omega,
 \frac{\partial \log u}{\partial z_4} \omega,
 \frac{\partial^2 u}{\partial z_5^2} \frac{1}{u} \omega
\end{equation}
and
\begin{equation}
\omega=\frac{dx \wedge dy}{xy},
 \frac{\partial \log u^{-1}}{\partial z_5} \omega,
 \frac{\partial \log u^{-1}}{\partial z_4} \omega,
 \frac{\partial^2 (u^{-1})}{\partial z_5^2} u \omega
\end{equation}
are bases of the twisted cohomology groups
and construct the connection matrices $d_z+\Omega$ and $d_z+\Omega^\vee$ with these bases
by Gr\"obner basis computation (see, e.g., \cite[chapter 6]{dojo-en}).
When 
$c = \left(\begin{array}{c}1/2\\ 1+\varepsilon\\ \varepsilon \end{array}\right)
$ and $z_1=z_2=z_3=1$ (we make this specialization with the same reason mentioned in Example \ref{example:gauss}),
we find by solving the secondary equation that
the intersection matrix of the cohomology group
is a scalar multiple of the matrix
\begin{equation}\label{eqn:4.4}
I=
\begin{pmatrix}
 1& \frac{  -  4{\ve}+ 1}{ 8{z}_{5}}& 0& \frac{    4   \ve^{ 2} +  3  \ve- 1}{  8{z}_{5}^{ 2} } \\

\frac{   4  \ve- 1}{ 8{z}_{5}}& \frac{  -  4   \ve^{ 2} +    \ve}{  8{z}_{5}^{ 2} }& 0& r_{24}\\

0& 0& 0& r_{34}\\

\frac{    4   \ve^{ 2} -  5  \ve+ 1}{ 8 {z}_{5}^{ 2} }&r_{42}&r_{43}&r_{44}
\end{pmatrix}
\end{equation}
where $r_{ij}$ are rational functions in the $z_4$, $z_5$ and $\ve$.
Note that this $I$ gives an example that $I_{ch}$ is not necessarily a constant matrix with respect to $z$.
By \cite[Theorem 8.1]{MH},
we can see that
the self intersection number of the first form is
\begin{equation}  \label{eq:self-intersection45}
  \left\langle \frac{dx \wedge dy}{xy}, \frac{dx \wedge dy}{xy} 
  \right\rangle \  |_{z_1=z_2=z_3=1}
  = \frac{32}{1-16\varepsilon^2} (2 \pi \sqrt{-1})^2.
\end{equation}
Thus, the intersection matrix is equal to
$  (2 \pi \sqrt{-1})^2 \frac{32}{1-16\varepsilon^2} I$,
which gives quadratic relations of GKZ hypergeometric series.
The whole intersection matrix of the twisted cohomology group is posted  
at our web page\footnote{
\url{http://www.math.kobe-u.ac.jp/OpenXM/Math/intersection/shiga3g2-imat45.rr}  ($\varepsilon$ is $b$ in this data. Risa/Asir program.)
}.

Let us show a quadratic relation by
taking a restriction to $\{ z_1=z_2=z_3=z_5=1\}$ (\cite{NS01}).
In this case,
we have 3 independent series
\begin{equation}
\varphi_{345}(1,1,1,z_4,1;c(\varepsilon))=z_4^{-\frac{1}{2}-2\varepsilon}\sum_{m,n\geq 0}\frac{(z_4^{-2})^{m+n}}{\Gamma(1+\varepsilon+n)\Gamma(\frac{1}{2}-2\varepsilon-2m-2n)\Gamma(1+\varepsilon+m)m!n!}
\end{equation}

\begin{align}
 &\varphi_{245}(1,1,1,z_4,1;c(\varepsilon))=\varphi_{134}(1,1,1,z_4,1;c(\varepsilon))\nonumber\\
=&z_4^{-\frac{1}{2}}\sum_{m,n\geq 0}\frac{(z_4^{-2})^{m+n}}{\Gamma(1-\varepsilon+n)\Gamma(\frac{1}{2}-2m-2n)\Gamma(1+\varepsilon+m)m!n!}
\end{align}

\begin{equation}
\varphi_{124}(1,1,1,z_4,1;c(\varepsilon))=z_4^{-\frac{1}{2}+2\varepsilon}\sum_{m,n\geq 0}\frac{(z_4^{-2})^{m+n}}{\Gamma(1-\varepsilon+n)\Gamma(\frac{1}{2}+2\varepsilon-2m-2n)\Gamma(1-\varepsilon+m)m!n!}.
\end{equation}
From our intersection matrix of the twisted cohomology group,
we obtain the following quadratic relation (\cite[Theorem 8.1]{MH}):
\begin{align}
 &\frac{1}{2}\left\{ \frac{\pi^3}{\sin^2\pi\varepsilon\cos\pi(2\varepsilon)}\varphi_{345}(1,1,1,z_4,1;c(\varepsilon))\varphi_{345}(1,1,1,z_4,1;-c(\varepsilon))\right.\nonumber\\
 &-\frac{2\pi^3}{\sin^2\pi\varepsilon}\varphi_{245}(1,1,1,z_4,1;c(\varepsilon))\varphi_{245}(1,1,1,z_4,1;-c(\varepsilon))\nonumber\\
 &\left. +\frac{\pi^3}{\sin^2\pi\varepsilon\cos\pi(2\varepsilon)}\varphi_{124}(1,1,1,z_4,1;c(\varepsilon))\varphi_{124}(1,1,1,z_4,1;-c(\varepsilon))\right\}\nonumber\\
=&\frac{32}{(1-16\varepsilon^2)}.\label{K3QuadRel1}
\end{align}

\noindent
When $\varepsilon=0$, the integral $f(z)$ reduces to the integral discussed in the context of mirror symmetry (\cite{NS01}). 
It satisfies a rank $3$ system. 
We are interested in this case.
Once we obtain the quadratic relation, we can take a limit 
$\varepsilon \rightarrow 0$ and get a quadratic relation for the case
of the resonant parameter value.
However, the set of the naive limits of $\varphi_{ijk}$'s does not give 
a basis of the solution space of the rank $3$ system.
We put $\alpha=\frac{\pi^3}{\sin^2\pi\varepsilon\cos\pi(2\varepsilon)}$, $\beta=-\frac{\pi^3}{\sin^2\pi\varepsilon}$,
\begin{equation}
\Phi_\ve=\scalebox{0.9}{$\Bigl(\varphi_{345}(1,1,1,z_4,1;c(\varepsilon)),\varphi_{245}(1,1,1,z_4,1;c(\varepsilon)),\varphi_{134}(1,1,1,z_4,1;c(\varepsilon)),\varphi_{124}(1,1,1,z_4,1;c(\varepsilon))\Bigr)$}
\end{equation}
and 
\begin{equation}
\Phi^\vee_\ve=\scalebox{0.85}{$\Bigl(\varphi_{345}(1,1,1,z_4,1;-c(\varepsilon)),\varphi_{245}(1,1,1,z_4,1;-c(\varepsilon)),\varphi_{134}(1,1,1,z_4,1;-c(\varepsilon)), \varphi_{124}(1,1,1,z_4,1;-c(\varepsilon))\Bigr).$}
\end{equation}
Let us construct a set of linearly independent solutions
by the Frobenius method from $\Phi_\varepsilon$ and $\Phi^\vee_\varepsilon$.
When $\varepsilon \not= 0$, the twisted period relation (\ref{K3QuadRel1}) can be written in the following form:
\begin{equation}
\Phi_\ve
\begin{pmatrix}
\alpha&0&0&0\\
0&\beta&0&0\\
0&0&\beta&0\\
0&0&0&\alpha
\end{pmatrix}
{}^t\Phi^\vee_\ve
=\frac{64}{(1-16\varepsilon^2)}.
\end{equation}
In order to obtain linearly independent solutions,
we introduce a $4\times 4$ matrix $Q$ defined by
\begin{equation}
Q=
\begin{pmatrix}
1&-\ve^{-1}&-\ve^{-1}&\ve^{-2}\\
0&\ve^{-1}&0&-\ve^{-2}\\
0&0&\ve^{-1}&-\ve^{-2}\\
0&0&0&\ve^{-2}
\end{pmatrix}.
\end{equation}
This matrix $Q$ is chosen so that the set
$$ \lim_{\varepsilon \rightarrow 0} \varphi_{345},\  
 \lim_{\varepsilon \rightarrow 0} (\varphi_{245}-\varphi_{345})/\varepsilon,\ 
 \lim_{\varepsilon \rightarrow 0} (\varphi_{134}-\varphi_{345})/\varepsilon,\ 
 \lim_{\varepsilon \rightarrow 0} (\varphi_{345}-\varphi_{245}-\varphi_{134}+\varphi_{124})/\varepsilon^2 
$$
is a basis of the solutions when $\varepsilon=0$.
This method is called the Frobenius method
(see, e.g., \cite[p.22, p.145]{SST}).
We put $\tilde{\Phi}_\ve=\Phi_\ve Q$ and $\tilde{\Phi}^\vee_\ve=\Phi^\vee_\ve Q$. A simple computation shows that $\tilde{\Phi}=\tilde{\Phi}_\ve\restriction_{\varepsilon\rightarrow 0}$ and $\tilde{\Phi}^\vee=\tilde{\Phi}^\vee_\ve\restriction_{\varepsilon\rightarrow 0}$ are convergent and we have a limit formula:
\begin{equation} \label{eq:quad0}
\tilde{\Phi}
\begin{pmatrix}
4\pi^3&0&0&\pi\\
0&0&\pi&0\\
0&\pi&0&0\\
\pi&0&0&0
\end{pmatrix}
{}^t\tilde{\Phi}^\vee
=64.
\end{equation}
The components of the vector ${\tilde \Phi}$ are as follows:
\begin{equation}
{\tilde \Phi}_1=
\frac{1}{\sqrt{\pi } \sqrt{z_4}}
\left(1+\frac{3}{2 z_4^2}+O\left(\left(\frac{1}{z_4}\right)^4\right) \right),
\end{equation}
\begin{eqnarray}
{\tilde \Phi}_2
&=& {\tilde \Phi}_3 =
\frac{1}{\sqrt{\pi } \sqrt{z_4}} (\phi'_{20} + (\log z_4) \phi'_{21}),
\end{eqnarray}
\begin{eqnarray}
{\tilde \Phi}_4  
&=&
\frac{1}{\sqrt{\pi } \sqrt{z_4}}(\phi'_{40}+(\log z_4) \phi'_{41} +(\log z_4)^2 \phi'_{42}).\end{eqnarray}
\noindent
The components of the vector ${\tilde \Phi}^\vee$ are as follows:
\begin{equation}
{\tilde \Phi}_1^\vee=\frac{2 \sqrt{z_4}}{\sqrt{\pi }}\left(1-\frac{1}{2 z_4^2}+O\left(\left(\frac{1}{z_4}\right)^4\right)\right),
\end{equation}
\begin{eqnarray}
{\tilde \Phi}_2^\vee&=&{\tilde \Phi}_3^\vee=\frac{2 \sqrt{z_4}}{\sqrt{\pi }}(\phi_{20} + (\log z_4)\phi_{21}),
\end{eqnarray}
\begin{eqnarray}
{\tilde \Phi}_4^\vee&=&\frac{2 \sqrt{z_4}}{\sqrt{\pi }}(\phi_{40} + (\log z_4)\phi_{41}+(\log z_4)^2 \phi_{42}).
\end{eqnarray}
Here, $\phi_{ij}$ and $\phi'_{ij}$ are power series in $z_4^{-2}$.
We would like to note that coefficients of these power series are very complicated,
but these satisfy a simple functional identity (\ref{eq:quad0}).
\mchange{Interested readers may} refer to our Mathematica programs {\tt tperiod.m} and {\tt fancy.m}\footnote{ \url{http://www.math.kobe-u.ac.jp/OpenXM/Math/intersection}}
to obtain explicit expressions of coefficients of these series.

As for $\tilde{\Phi}_1$, it can be related to Thomae's and Gau\ss' hypergeometric series by \mchange{a} simple transformation
\begin{equation}
 \pi^{1/2} z_4^{1/2} {\tilde \Phi}_1(z_4)= {}_3 F_2\left(\substack{1/4,2/4,3/4 \\1,1}; 16/z_4^2\right) =
\left({}_2F_1\left(\substack{1/8,3/8\\ 1};16/z_4^2\right)\right)^2
\end{equation}
as was remarked in \cite{NS01}. Here, the second equality is the result of Clausen's identity.

\section*{Appendix: A lemma on a stalk of a direct image}
In this appendix, we prove the following lemma.

\begin{lem}\label{LemmaA}
Let $X,F$ be real manifolds, let $p:X\times F\rightarrow X$, \mchange{$q:X\times F\rightarrow F$} be the canonical projections, and let $\mathcal{L}$ be a $\Q$-local system on $F$. Then, for any $x\in X$ and any integer $k$, there is a canonical isomorphism 
\begin{equation}\label{StalkFormula}
\left(R^kp_*(\mchange{q^{-1}}\mathcal{L})\right)_x\simeq H^k\left( F, \mathcal{L}\right)
\end{equation}
under the natural identification $F\simeq \{ x\}\times F$.
\end{lem}

\begin{proof}
Since $R^kp_*(\mchange{q^{-1}}\mathcal{L})$ is the sheaf associated to the presheaf $U\mapsto \Homo^k\left( p^{-1}(U),\mchange{q^{-1}}\mathcal{L}\right)$, we see that the left-hand side of (\ref{StalkFormula}) is equal to $\underset{x\in U}{\varinjlim}\Homo^k\left( p^{-1}(U),\mchange{q^{-1}}\mathcal{L}\right)$. Since $X$ is a real manifold, we can find a fundamental system of neighborhoods of $x$ consisting of balls in an Euclidian space. When $U$ is a ball, by K\"unneth formula, we have a sequence of isomorphisms $\Homo^k\left( p^{-1}(U),\mchange{q^{-1}}\mathcal{L}\right)\simeq \Homo^0\left( U,\Q\right)\otimes_\Q\Homo^k\left( F,\mathcal{L}\right)\simeq\Homo^k\left( F,\mathcal{L}\right).$
\end{proof}


\end{document}